\documentclass[12pt,twoside]{amsart}
\usepackage{amssymb}
\usepackage{verbatim}
\usepackage{amsmath}
\usepackage{mathrsfs}
\usepackage{bm}
\usepackage{a4wide}
\usepackage[latin1]{inputenc}
\usepackage[T1]{fontenc}
\usepackage{times}
\usepackage{amssymb,latexsym}
\usepackage{enumerate}
\usepackage{pict2e}
\usepackage{color}
\usepackage[colorlinks,linkcolor=black,citecolor=black,urlcolor=black]{hyperref}

\makeatletter
\newcommand{\sumprime}{\if@display\sideset{}{'}\sum%
            \else\sum'\fi}
\makeatother

\makeatletter
\DeclareRobustCommand{\intprod}{%
  \mathbin{\mathpalette\int@prod{(0.1,0)(0.9,0)(0.9,0.8)}}%
}
\DeclareRobustCommand{\intprodr}{%
  \mathbin{\mathpalette\int@prod{(0.1,0.8)(0.1,0)(0.9,0)}}}

\newcommand{\int@prod}[2]{%
  \begingroup
  \sbox\z@{$\m@th#1+$}%
  \setlength\unitlength{\wd\z@}%
  \begin{picture}(1,1)
  \roundcap
  \polyline#2
  \end{picture}%
 \endgroup
}
\makeatother

\begin{document}

\numberwithin{equation}{section}

\newtheorem{theorem}{Theorem}[section]
\newtheorem{proposition}[theorem]{Proposition}
\newtheorem{conjecture}[theorem]{Conjecture}
\def\theconjecture{\unskip}
\newtheorem{corollary}[theorem]{Corollary}
\newtheorem{lemma}[theorem]{Lemma}
\newtheorem{observation}[theorem]{Observation}
\newtheorem{definition}{Definition}
\numberwithin{definition}{section} 
\def\theremark{\unskip}
\newtheorem{kl}{Key Lemma}
\def\thekl{\unskip}
\newtheorem{question}{Question}
\def\thequestion{\unskip}
\newtheorem{example}{Example}
\def\theexample{\unskip}
\newtheorem{problem}{Problem}

\newtheoremstyle{mythm1}{2ex plus 1ex minus .2ex}{2ex plus 1ex minus .2ex}     {\normalfont}{}{\bfseries}{.}{0.7em}{} 
\theoremstyle{mythm1}
\newtheorem*{remark}{Remark}

\thanks{The first author is supported by National Natural Sciense Foundation of China, No. 12271101. The second author is supported by China Postdoctoral Science Foundation, No. 2024M750487.}

\address{School of Mathematical Sciences, Fudan University, Shanghai, 20043, China}

\email{boychen@fudan.edu.cn}

\address{School of Mathematical Sciences, Fudan University, Shanghai, 200433, China}

\email{ypxiong@fudan.edu.cn}

\title[Real Variable Things in Bergman Theory]{Real Variable Things in Bergman Theory}

\author{Bo-Yong Chen and Yuanpu Xiong}
\date{}

\maketitle

\begin{center}
\emph{In memory of Joseph John Kohn}
\end{center}

\begin{abstract}
 In this article,  we investigate the connection between certain real variable things and the Bergman theory.  We first use Hardy-type inequalities to give an $L^2$ Hartogs-type extension theorem and an $L^p$ integrability theorem for the Bergman kernel $K_\Omega(\cdot,w)$.  We then use the Sobolev-Morrey inequality to show the absolute continuity of  Bergman kernels on planar domains with respect to logarithmic capacities.  Finally,  we give  lower bounds of the minimum $\kappa(\Omega)$ of the Bergman kernel $K_\Omega(z)$   in terms of the interior capacity radius for planar domains and the volume density for bounded pseudoconvex domains in $\mathbb C^n$.   As a consequence,  we show that $\kappa(\Omega)\ge c_0 \lambda_1(\Omega)$ holds on planar domains,  where $c_0$ is a numerical constant and  $\lambda_1(\Omega)$ is the first Dirichlet eigenvalue of $-\Delta$. 
\newline
\newline
\noindent Keywords: Bergman kernel, Hardy inequality, Sobolev-Morrey inequality, the first eigenvalue, capacity
\newline
\newline
\noindent 2020 Mathematics Subject Classification: 32A25, 35P15, 30C85.
\end{abstract}

\section{Introduction}

It is well-known that complex analysis can be used to give beautiful  proofs of real theorems (see e.g.,  \cite{LaxZalcman}).    On the other hand,  real methods,  e.g.,  the $L^2$ theory of the Cauchy-Riemann equation pioneered by Kohn and H\"ormander,  provide powerful tools for complex analysis.  The present article may be viewed as a continuation of this tradition.  

One of the landmark result in several complex variables is the Hartogs extension theorem, i.e.,  
if $\Omega$ is a domain in $\mathbb C^n$ with $n\ge 2$ and $E$ a compact subset in $\Omega$ such that $\Omega\backslash E$ is connected,  then every holomorphic function on $\Omega\backslash E$ can be extended holomorphically to $\Omega$.  Here we would like to give an alternative proof of this theorem based on the following  classical Hardy inequality 
\begin{equation}\label{eq:Hardy_1}
\frac{N-2}4 \int_{\mathbb R^N} \frac{|\phi|^2}{|x|^2} \le \int_{\mathbb R^N}  |\nabla \phi|^2,\ \ \ \forall\,\phi\in C^\infty_0(\mathbb R^N),\ N\ge 3.
\end{equation}
Since the proof of \eqref{eq:Hardy_1} uses essentially the fact that $\mathbb R^N$ is hyperbolic for $N\ge 3$ (note that $-|x|^{2-N}$ provides a negative subharmonic function on $\mathbb R^N$),    it turns out that the basic reason for the Hartogs extension phenomenon is the hyperbolicity of $\mathbb C^n$  when $n\ge 2$.  

Actually,  we can prove the following $L^2$ Hartogs-type theorem whenever $E$ is not necessarily compact.   

\begin{theorem}\label{th:Hartogs_L2}
Let $\Omega$ be a domain in $\mathbb C^n$ ($n\geq2$) and $E$ a closed set in $\mathbb C^n$ such that
\begin{enumerate}
\item[$(1)$] there exists $r>0$ such that $E_r:=\{z\in \mathbb C^n: d(z,E)\le r\}\subset \Omega$;
\item[$(2)$] there exist an affine-linear subspace $H\subset \mathbb R^{2n}=\mathbb C^n$ of real codimension $\ge 3$ and a number $R>0$ such that 
$$
E\subset H_R:=\{z\in \mathbb C^n: d(z,H)<R\};
$$
\item[$(3)$] $\Omega\backslash E$ is connected.
\end{enumerate}
Then every $L^2$ holomorphic function defined on $\Omega\backslash E$ can be extended to an $L^2$ holomorphic function on $\Omega$.  In other words,  the Bergman space $A^2(\Omega\setminus E)$ of\/ $\Omega\setminus E$ coincides with $A^2(\Omega)$. 
\end{theorem}

A large literature exists for different proofs and generalizations of the Hartogs extension theorems (cf.  \cite{Bochner1},  \cite{Bochner2},  \cite{Bochner3},  \cite{BDS1},  \cite{BDS2},  \cite{Ehrenpreis},   \cite{Ohsawa},  \cite{Range}).  Among them we mention in particular the theorem of Bochner \cite{Bochner2} that every $L^2$ holomorphic function defined on  $D\times i\mathbb R^n$ can be extended to an $L^2$ holomorphic function on  $(\text{convex hull of\ }D)\times i\mathbb R^n$ for a domain $D\subset \subset \mathbb R^n$,  which is closely related to Theorem \ref{th:Hartogs_L2}.  

The Hardy inequality for a domain $\Omega\subsetneq \mathbb R^N$ is given as follows
\begin{equation}\label{eq:Hardy_2}
c^2 \int_\Omega |\phi|^2/\delta_\Omega^2 \le \int_\Omega |\nabla \phi|^2,\ \ \ \forall\,\phi\in C^\infty_0(\Omega),
\end{equation}
where $\delta_\Omega$ denotes the boundary distance.   The Hardy constant $h(\Omega)$ of $\Omega$ is defined by
\[
h(\Omega)^2:=\inf\left\{\frac{\int_\Omega|\nabla\phi|^2}{\int_\Omega\phi^2/\delta_\Omega^2}: \phi\in C_0^\infty(\Omega)\setminus \{0\} \right\}.
\]
It is well known that $h(\Omega)>0$ when $\Omega$ is a bounded domain with Lipschitz boundary and $h(\Omega)=1/2$ when $\Omega$ is a convex domain. On the other hand, there exists non-convex sectors $\Omega\subset\mathbb{C}$ such that $h(\Omega)=1/2$ (cf. Davies \cite{Davies95}).

 Let  $K_\Omega(z,w)$ denote the Bergman kernel on a domain $\Omega\subset \mathbb C^n$.   Following \cite{Chen17Hyperconvex},  we define the integrability index of $K_\Omega(\cdot,w)$ by
\[
\beta(\Omega):=\sup\{\beta\geq2:K_\Omega(\cdot,w)\in{L^\beta}(\Omega),\ \ \ \forall\,w\in\Omega\}.
\]
This concept is motivated by a celebrated conjecture of Brennan \cite{Brennan}.  An equivalent statement of this conjecture is that $\beta(\Omega)\ge 4$ holds for any simply-connected domain $\Omega\subsetneq \mathbb C$. It is also  interesting to find reasonable lower bounds of $\beta(\Omega)$ for arbitrary planar domains.  
  
\begin{theorem}\label{th:Bergman_int}
If\/ $\Omega$ is  a domain in $\mathbb C$,  then  $\beta(\Omega)\geq\frac{6}{3-h(\Omega)}$.
\end{theorem}

Next recall the following Sobolev-Morrey inequality (see, e.g., \cite{Cianchi08, Talenti94})
\begin{equation}\label{eq:Morrey}
\sup_\Omega|\phi|\leq C_0(q-2)^{-1/2}|\Omega|^{1/2-1/q}\left(\int_\Omega|\nabla \phi|^q\right)^{1/q},\ \ \ \forall\,\phi\in{C^\infty_0(\Omega)},\ q>2,
\end{equation}
where $C_0$ is a numerical constant. This inequality can be used to obtain an $L^p$ estimate for the $\bar{\partial}$-equation, which goes back to Berndtsson \cite{Berndtsson92}. This combined with parameter dependence of the $p$-Bergman kernel (cf. \cite{CZ22}) gives the following property on the absolute continuity with respect to logarithmic capacities.

\begin{theorem}\label{th:absolute_continuity}
Let $\Omega\subset\mathbb{C}$ be a bounded domain.  Let $0<r_0\ll1$ be given.  For any $\varepsilon>0$, there exists some $\gamma>0$ such that
\[
|K_{\Omega\setminus{E}}(z,w)-K_\Omega(z,w)|<\varepsilon
\]
holds for any closed subset $E\subset\mathbb{C}$ whose logarithmic capacity satisfies $\mathcal{C}_l(E)<\gamma$ and any $z,w\in \Omega$ with
 $$
 \min\{\delta_{\Omega\setminus{E}}(z),\delta_{\Omega\setminus{E}}(w)\}\geq{r_0}.
 $$
\end{theorem}

A well-known theorem of Carleson asserts that $K_{\Omega\setminus{E}}(z)=K_\Omega(z)$ if $E$ is a polar set, i.e., $\mathcal{C}_l(E)=0$. It is not difficult to see that Theorem \ref{th:absolute_continuity} implies Carleson's theorem. Under certain reasonable assumption,  we even have  an effective version of Theorem \ref{th:absolute_continuity}. 

\begin{theorem}\label{th:absolute_continuity_effective}
Let $\Omega\subset \mathbb C$ be a domain on which there exists  a negative continuous  subharmonic function $\rho$ on $\Omega$ with $-\rho\lesssim \delta_\Omega^c$ for some $c>0$.    Let $0<r_0\ll1$ be given.  For any  $0<\alpha<1/3$,  there exists a constant $C=V(\alpha,r_0,\Omega)>0$ such that for any closed subset $E\subset\mathbb{C}$ with $\mathcal{C}_l(E)\ll1$  and any $z,w\in \Omega$ with $\min\{\delta_{\Omega\setminus{E}}(z),\delta_{\Omega\setminus{E}}(w)\}\geq{r_0}$,   
\[
|K_{\Omega\setminus{E}}(z,w)-K_\Omega(z,w)|\leq C \left[\log\frac{r_0}{\mathcal{C}_l(E)}\right]^{-\alpha}.
\]
\end{theorem}

In general,  Theorem \ref{th:absolute_continuity} and Theorem \ref{th:absolute_continuity_effective} do not hold if capacities are replaced by volumes.  A simple counterexample is  $\Omega=2\mathbb D$ and $E$ is certain small tublar neighbourhood of  the $1/3$-Cantor set.

On the other hand,  for domains in $\mathbb{C}^n$ with $n\geq2$, we have the following 

\begin{proposition}\label{prop:Hartogs_absolute_continuity}
Let $\Omega$ be a domain in $\mathbb{C}^n$ ($n\geq2$) and $E_0$ a closed set in $\mathbb{C}^n$ which satisfies the conditions (1)-(3) in Theorem \ref{th:Hartogs_L2}. Then 
\[
|K_{\Omega\setminus{E}}(z,w)-K_\Omega(z,w)|\le  C |E|(K_\Omega(z)+K_\Omega(w))
\]
for any closed set $E\subset{E_0}$ and $z,w\in\Omega\setminus{E}$.  In particular, if $|E|=0$, then $K_{\Omega\setminus E}(z,w)=K_\Omega(z,w)$ for $z,w\in\Omega\setminus E$.
\end{proposition}

Finally, we recall the following Poincar\'{e}-Friderichs inequality
\begin{equation}\label{eq:Poincare_Friderichs}
\int_\Omega |\phi|^2 \leq \lambda^{-1} \int_\Omega|\nabla\phi|^2,\ \ \ \forall\,\phi\in C^\infty_0(\Omega).
\end{equation}
As is well-known, the optimal $\lambda$ is the first Dirichlet eigenvalue $\lambda_1(\Omega)$ of $-\Delta$ on $\Omega$ (w.r.t the Euclidean metric), i.e.,
$$
\lambda_1(\Omega)=\inf_{\phi\in C^\infty_0(\Omega)\setminus \{0\}} \frac{\int_\Omega |\nabla \phi|^2} {\int_\Omega |\phi|^2 }.  
$$
A classical topic in spectrum theory is to study the relation between $\lambda_1(\Omega)$ and global geometry of $\Omega\subset\mathbb{R}^n$.   A popular estimate is 
$$
\lambda_1(\Omega)\gtrsim_n |\Omega|^{-2/n}.
$$
 Under certain reasonable assumption on the boundary (e.g.,  $\partial \Omega$ is Lipschitz),  a better estimate holds:
 $$
\lambda_1(\Omega)\gtrsim \mathscr R(\Omega)^{-2},\ \ \ \mathscr R(\Omega): \text{inradius of } \Omega.
$$  
Lieb \cite{Lieb83} showed that for\/ {\it any} domain $\Omega\subsetneq \mathbb R^n$, 
\begin{equation}\label{eq:Lieb}
\lambda_1(\Omega) \ge \frac{n}{4r^2} \inf_{x\in \Omega} \frac{| B(x,r)\setminus\Omega|}{|B(x,r)|},\ \ \ \forall\,  r>0.
\end{equation}
Maz'ya-Shubin \cite{MS05} further showed that for $\Omega\subsetneq\mathbb{R}^n$ with $n\geq3$,
\begin{equation}\label{eq:MS}
\lambda_1(\Omega) \asymp_{\alpha} \mathscr R_{W,\alpha}(\Omega)^{-2},\ \ \ \forall\, 0<\alpha<1.
\end{equation}
Here $A\asymp_\alpha B$ if and only if both $A/B$ and $B/A$ are bounded by a constant depending only on $\alpha$,  $\mathscr R_{W,\alpha}$   denotes 
   the {\it interior capacity radius} (w.r.t $\alpha$)  given by 
\[
\mathscr{R}_{W,\alpha}(\Omega):= \sup\left\{r:\exists\,B(x,r)\ \text{s.t.\ } \frac{\mathcal{C}_W(\overline{B(x,r)}\setminus\Omega)}{\mathcal{C}_W (B(x,r))} \le \alpha \right\},
\]
and $\mathcal{C}_W$ stands for the Wiener capacity,  i.e., for a compact set $F\subset \mathbb R^n$,
\[
\mathcal{C}_W(F):=\inf\left\{\int_{\mathbb R^n}|\nabla\phi|^2:\phi\in C^\infty_0(\mathbb R^n), \phi|_F=1\right\}.
\]

In order to obtain analogous estimates in Bergman theory, we introduce the following quantity
$$
\kappa(\Omega):=\inf_{z\in \Omega} K_\Omega(z),\ \ \ \Omega\subsetneq \mathbb C.  
$$
The original interest of  $\kappa(\Omega)$ arises from the study of Riemann map $f:\mathbb D\rightarrow \Omega$, where $\Omega\subsetneq\mathbb{C}$ is a simply-connected domain.  Transformation rule of the Bergman kernel yields
$$
K_\mathbb D(z)= K_\Omega(f(z))|f'(z)|^2,
$$
so that the Bloch semi-norm of $f$ satisfies
$$
\|f\|_{\mathcal B}:=\sup_{z\in \mathbb D}\left\{(1-|z|^2)|f'(z)|\right\} = \frac1{\sqrt{\pi \kappa(\Omega)}}.
$$
Recently, Nazarov \cite{Nazarov12} found an interesting Bergman kernel approach to the Bourgain-Milman estimate
\[
|\mathcal{K}|\cdot|\mathcal{K}^\circ|\geq c^n\frac{4^n}{n!}
\]
where $\mathcal{K}$ is an origin-symmetric bounded convex body in $\mathbb{R}^n$, $\mathcal{K}^\circ:=\{t\in\mathbb{R}^n:\langle{x,t}\rangle\leq1\}$, and $c>0$ is a numerical constant (see also B{\l}ocki \cite{Blocki15} and Berndtsson \cite{Berndtsson22}). More precisely, Nazarov considered the tube domain $T_{\mathcal{K}}:=\{x+\mathrm{i}y:x\in\mathbb{R}^n,\ y\in K\}$. He first obtained an upper bound
\[
K_{T_{\mathcal{K}}}(0)\leq 2^nn!|\mathcal{K}^\circ|/|\mathcal{K}|,
\]
then used H\"{o}rmander's $L^2$ estimates for $\bar{\partial}$ to get a lower bound
\[
K_{T_{\mathcal{K}}}(0)\geq 8^nc^n|\mathcal{K}|^{-2},
\]
from which the Bourgain-Milman estimate immediately follows. Note that $K_{T_{\mathcal{K}}}(0)=\kappa(T_{\mathcal{K}})$, for the Bergman kernel is a convex function on the convex tube $T_{\mathcal{K}}$ and $\mathcal{K}$ is origin-symmetric.

Similar as \cite{MS05}, given a domain $\Omega\subsetneq\mathbb{C}$ we define 
\[
\mathscr{R}_{L,\alpha} (\Omega):=\sup\left\{r>0:{\mathcal{C}_l(\overline{\Delta(z,s)}\setminus\Omega)} \leq\alpha s,\ \text{for some }z\in\Omega\text{ and any }s\in(0,r)\right\},
\]
where $0<\alpha<1$ and $\Delta(z,s)=\{\zeta\in \mathbb C: |\zeta-z|<s\}$.   Again "$\mathcal{C}_l$" means the logarithmic capacity. We have the following Maz'ya-Shubin type estimate for $\kappa(\Omega)$.

\begin{theorem}\label{th:minimum_value}
For any domain $\Omega\subsetneq\mathbb C$,  we have
\[
\kappa(\Omega)\geq\frac{\alpha^2}{\pi\mathscr{R}_{L,\alpha}(\Omega)^2}.
\]
\end{theorem}

For a domain $\Omega\subsetneq\mathbb{C}$, the argument of Maz'ya-Shubin also gives an estimate of $\lambda_1(\Omega)$ by using $\mathscr{R}_{L,\alpha}$. For reader's convenience and for our purpose, we present a proof of the upper estimate in a slightly weak form in the appendix (see \S\,\ref{sec:appendix}), which will be used to give the following relationship between $\kappa(\Omega)$ and $\lambda_1(\Omega)$\footnote{The authors are indebted to Professor Siqi Fu for proposing this question during the Shenzhen International Workshop on SCV and CR Geometry in 2024.}.

\begin{corollary}\label{cor:Minimum_eigenvalue}
There exists a numerical constant $c_0$ such that
\[
\kappa(\Omega)\geq c_0\lambda_1(\Omega)
\]
for any domain $\Omega\subsetneq\mathbb C$.
\end{corollary}

It is natural to ask
\begin{question}
Does there exist some numerical constant $C_0>0$ such that
\[
\kappa(\Omega)\leq C_0\lambda_1(\Omega)
\]
for any domain $\Omega\subsetneq\mathbb C$?
\end{question}

For high-dimensional cases, Theorem \ref{th:minimum_value} together with the Ohsawa-Takegoshi extension theorem yield the following Lieb-type estimate for $\kappa(\Omega)$.

  \begin{theorem}\label{th:minimum_value_2}
Let   $\Omega\subset \mathbb{C}^n$ be a bounded pseudoconvex domain.  Define 
$$
\mathscr V(\Omega):=\inf_{w\in \partial \Omega,r>0} \frac{|B(w,r)\setminus\Omega|}{|B(w,r)|}.  
$$
Then
$$
\kappa(\Omega) \ge \frac{C_n \mathscr V(\Omega)^2}{\mathscr R(\Omega)^2\mathscr D(\Omega)^{2n-2}},
$$
where $\mathscr R(\Omega)$ and $\mathscr D(\Omega)$ denote the inradius and the diameter of\/ $\Omega$ respectively.
\end{theorem}

\section{Hardy type inequalities and $L^2$ estimates for the $\bar{\partial}$-equation}

\subsection{An a priori estimate}

Let $\Omega$ be a domain in $\mathbb{C}^n=\mathbb{R}^{2n}$. Suppose the following Hardy-type inequality
\begin{equation}\label{eq:Hardy}
\int_\Omega \phi^2 \eta\le \int_\Omega |\nabla \phi|^2,\ \ \ \forall\,\phi\in C^\infty_0(\Omega)
\end{equation}
holds, where $\eta\geq0$ is a measurable function. The special case when $\eta=\frac{(n-1)^2}{|z|^2}$ is the standard Hardy inequality \eqref{eq:Hardy_1}.

For $z_k=x_k+\mathrm{i}y_k$ ($1\leq k\leq n$), we write
\begin{equation}\label{eq:complex_derivative}
\left|\frac{\partial\phi}{\partial\bar{z}_k}\right|^2
= \frac{1}{4}\left(\left|\frac{\partial\phi}{\partial{x_k}}\right|^2+\left|\frac{\partial\phi}{\partial{y_k}}\right|^2\right) - \frac{\mathrm{i}}{2}\left(\frac{\partial\phi}{\partial{x_k}}\frac{\partial\bar{\phi}}{\partial{y_k}}-\frac{\partial\phi}{\partial{y_k}}\frac{\partial\bar{\phi}}{\partial{x_k}}\right).
\end{equation}
Since
\[
\int_\Omega\frac{\partial\phi}{\partial{x_k}}\frac{\partial\bar{\phi}}{\partial{y_k}}
=-\int_\Omega\frac{\partial^2\phi}{\partial{y_k}\partial{x_k}}\bar{\phi}
=\int_\Omega\frac{\partial\phi}{\partial{y_k}}\frac{\partial\bar{\phi}}{\partial{x_k}},
\]
it follows from \eqref{eq:Hardy} and \eqref{eq:complex_derivative} that
\begin{equation}\label{eq:Hardy_complex}
\sum^n_{k=1}\int_\Omega\left|\frac{\partial\phi}{\partial\bar{z}_k}\right|^2
= \frac{1}{4}\int_\Omega|\nabla\phi|^2
\geq \frac{1}{4}\int_\Omega|\phi|^2\eta,
\end{equation}

Denote by $\mathcal D_{(0,q)}(\Omega)$ the set of smooth $(0,q)$-forms with compact supports in $\Omega$. Let $\varphi$ be a real-valued smooth function on $\Omega$ and $L^2_{(0,q)}(\Omega,\varphi)$ the completion of $\mathcal D_{(0,q)}(\Omega)$ with respect to the weighted $L^2$ inner product
\[
(u,v)_\varphi={\sum_{|J|=q}}'\int_\Omega u_J\overline{v}_Je^{-\varphi},
\]
where $u=\sum'_{|J|=q}u_Jd\bar{z}_J$ and $v=\sum'_{|J|=q}v_Jd\bar{z}_J$. Then $\bar{\partial}: L^2_{(0,q)}(\Omega,\varphi)\rightarrow L^2_{(0,q+1)}(\Omega,\varphi)$ gives a  densely defined and closed operator such that $\mathrm{Dom}\,(\bar{\partial})$ consists of all $u\in{L^2_{(0,q)}(\Omega,\varphi)}$ with $\bar{\partial}u\in{L^2_{(0,q+1)}(\Omega,\varphi)}$ in the sense of distributions. Let $\bar{\partial}_\varphi^\ast$ denote the Hilbert space adjoint of $\bar{\partial}$ (write $\bar{\partial}^\ast:=\bar{\partial}^\ast_0$). Let $q=1$. For any $\omega=\sum^n_{j=1}\omega_j d\bar{z}_j\in \mathcal D_{(0,1)}(\Omega)$, the Morrey-Kohn-H\"{o}rmander formula asserts that
\[
\sum^n_{j,k=1}\int_\Omega\frac{\partial^2\varphi}{\partial{z_j}\partial\bar{z}_k}\omega_j\bar{\omega}_ke^{-\varphi}+\sum^n_{j,k=1}\int_\Omega\left|\frac{\partial \omega_j}{\partial\bar{z}_k}\right|^2e^{-\varphi}=\int_\Omega|\bar{\partial}\omega|^2e^{-\varphi}+\int_\Omega|\bar{\partial}^\ast_\varphi \omega|^2e^{-\varphi}.
\]
This combined with \eqref{eq:Hardy_complex} gives

\begin{proposition}
If \eqref{eq:Hardy} holds , then
\begin{equation}\label{eq:Hardy_MKH}
\int_\Omega|\omega|^2\eta
\leq4\int_\Omega|\bar{\partial}\omega|^2+4\int_\Omega|\bar{\partial}^*\omega|^2,\ \ \ \forall\,\omega\in\mathcal{D}_{(0,1)}(\Omega).
\end{equation}
\end{proposition}

\subsection{The case $\Omega=\mathbb{C}^n$}

Let us recall the following classical result.

\begin{lemma}[Hardy inequality]\label{lm:Hardy}
Let $H\subset \mathbb R^{N}$ be an affine-linear subspace with codimension   $m\ge 3$.   Define $d_H=d(\cdot,H)$.  Then we have
\[
\frac{(m-2)^2}4 \int_{\mathbb R^{N}} \frac{|\phi|^2}{d_H^2} \le \int_{\mathbb R^{N}} |\nabla \phi|^2,\ \ \ \forall\,\phi\in C^\infty_0(\mathbb{R}^N).
\]
\end{lemma}

\begin{proof}
For the sake of completeness,  we still provide a proof here.  We may assume that $\phi$ is real-valued and $H=\{x'=0\}$ where $x'=(x_1,\cdots,x_m)$.  Then we have $d_H(x)=|x'|$ and the function $\psi(x)=\psi(x')=-|x'|^{2-m}$ is subharmonic on $\mathbb R^m$,  hence is subharmonic on $\mathbb R^N$ (in particular,  $\mathbb R^N$ is non-parabolic).  Thus
$$
0\le \int_{\mathbb R^N} \phi^2\cdot \frac{\Delta \psi}{-\psi}=- \int_{\mathbb R^N} \nabla \psi \cdot \nabla\left(\frac{\phi^2}{-\psi}\right)=2 \int_{\mathbb R^N} \frac{\phi}{\psi}\cdot\nabla \psi\cdot\nabla \phi-\int_{\mathbb R^N} \frac{\phi^2}{\psi^2}\cdot |\nabla \psi|^2,
$$
so that
\begin{eqnarray*}
\int_{\mathbb R^N} \frac{\phi^2}{\psi^2}\cdot |\nabla \psi|^2 & \le & -2\int_{\mathbb R^N} \frac{\phi}{\psi}\cdot\nabla \psi\cdot\nabla \phi\\
& \le & \frac12 \int_{\mathbb R^N} \frac{\phi^2}{\psi^2}\cdot |\nabla \psi|^2 + 2\int_{\mathbb R^N} |\nabla \phi|^2,
\end{eqnarray*}
that is
$$
\int_{\mathbb R^N} \frac{\phi^2}{\psi^2}\cdot |\nabla \psi|^2\le 4\int_{\mathbb R^N} |\nabla \phi|^2.  
$$
On the other hand,  a straightforward calculation gives $|\nabla \psi|^2/\psi^2=(m-2)^2 |x'|^{-2}$,  hence we are done.
\end{proof}

In particular, \eqref{eq:Hardy} holds for $\Omega=\mathbb{C}^n=\mathbb{R}^{2n}$ ($n\geq2$) and $\eta=(m-2)^2d_H^{-2}/4$ ($m\geq3$), i.e.,
\begin{equation}\label{eq:Hardy_MKH_Cn}
\int_\Omega\frac{|\omega|^2}{d_H^2}\leq\frac{16}{(m-2)^2}\left(\int_\Omega|\bar{\partial}\omega|^2+\int_\Omega|\bar{\partial}^\ast \omega|^2\right),\ \ \ \forall\,\omega\in\mathcal{D}_{(0,1)}(\mathbb{C}^n).
\end{equation}
Since $\mathcal{D}_{(0,1)}(\mathbb{C}^n)$ is dense in $\mathrm{Dom}(\bar{\partial})\cap\mathrm{Dom}(\bar{\partial}^\ast)$ with respect to the graph norm
\begin{equation}\label{eq:graph_norm}
\omega\mapsto\left(\int_{\mathbb C^n}|\omega|^2+\int_{\mathbb C^n}|\bar{\partial}\omega|^2+\int_{\mathbb C^n}|\bar{\partial}^\ast \omega|^2\right)^{1/2},
\end{equation}
we see that \eqref{eq:Hardy_MKH_Cn} also holds for $\omega\in\mathrm{Dom}(\bar{\partial})\cap\mathrm{Dom}(\bar{\partial}^\ast)$. This leads to the following

\begin{proposition}\label{prop:L2_Cn}
Let $H\subset \mathbb R^{2n}=\mathbb C^n$ be an affine-linear subspace with real codimension $m\ge 3$.  Then for any  $\bar{\partial}-$closed $(0,1)$-form $v$ on $\mathbb C^n$ with $\int_{\mathbb C^n} |v|^2<\infty$ and $\int_{\mathbb C^n} |v|^2 d_H^2<\infty$,  there exists a function $u$ on $\mathbb C^n$ such that  $\bar{\partial}u=v$ and
\begin{equation}\label{eq:L2_estimate_Cn}
\int_{\mathbb C^n} |u|^2 \le \frac{16}{(m-2)^2} \int_{\mathbb C^n} |v|^2 d_H^2.
\end{equation}
\end{proposition}

\begin{proof}
The proof is a standard application of the duality argument. For $\omega\in \mathrm{Dom}(\bar{\partial}^*)\cap \mathrm{Ker\,}\bar{\partial}$, it follow from the Cauchy-Schwarz inequality and \eqref{eq:Hardy_MKH_Cn} that
\begin{eqnarray}\label{eq:T}
|(\omega,v)|^2 & \le & \int_{\mathbb{C}^n} \frac{|\omega|^2}{d_H^2} \cdot \int_{\mathbb{C}^n} |v|^2d_H^2\nonumber\\
& \le & \frac{16}{(m-2)^2}\int_{\mathbb{C}^n}|\bar{\partial}^\ast \omega|^2 \cdot \int_{\mathbb{C}^n} |v|^2d_H^2.
\end{eqnarray}
Given $\omega\in\mathrm{Dom}(\bar{\partial^\ast})$,  write $\omega=\omega_1+\omega_2\in(\mathrm{Ker}\,\bar{\partial})\oplus(\mathrm{Ker}\,\bar{\partial})^\perp$. We have
\[
\omega_2\in(\mathrm{Ker}\,\bar{\partial})^\perp=\overline{\mathrm{Im}\,\bar{\partial}^*}\subset\mathrm{Ker}\,\bar{\partial}^\ast,
\]
i.e., $\bar{\partial}^\ast\omega_2=0$, for $\bar{\partial}^*\bar{\partial}^*=0$ and $\mathrm{Ker}\,\bar{\partial}^*$ is a closed subspace of $L^2_{(0,1)}(\Omega)$. On the other hand, we have $(\omega_2,v)=0$ since $v\in\mathrm{Ker}\,\bar{\partial}$. Thus \eqref{eq:T} remains valid for all $\omega\in\mathrm{Dom}(\bar{\partial}^\ast)$. It follows that the linear functional
\[
T: \bar{\partial}^\ast \omega\mapsto(\omega,v),\ \ \ \omega\in\mathrm{Dom}(\bar{\partial}^\ast)
\]
is well-defined and satisfies
\[
\|T\|\leq\frac{16}{(m-2)^2} \cdot \int_{\mathbb{C}^n} |v|^2d_H^2.
\]
The Hahn-Banach theorem combined with the Riesz representation theorem yields a function $u$ on $\mathbb{C}^n$ satisfying \eqref{eq:L2_estimate_Cn} and
\[
(\omega,v)=(\bar{\partial}^\ast \omega,u),\ \ \ \forall\,\omega\in \mathrm{Dom}(\bar{\partial}^\ast).
\]
Namely, $\bar{\partial}u=v$ holds in the sense of distributions.
\end{proof}

\subsection{The case of planar domains}

If $n=1$, we have $\bar{\partial}\omega=0$ for any $(0,1)$-form $\omega$, so that \eqref{eq:Hardy_MKH} becomes
\begin{equation}\label{eq:Hardy_MKH_C}
\int_\Omega|\omega|^2\eta\leq4\int_\Omega|\bar{\partial}^\ast \omega|^2,\ \ \ \forall\,\omega\in\mathcal{D}_{(0,1)}(\Omega).
\end{equation}

In order to get an estimate for the canonical solution (i.e., the $L^2$ minimal solution) of the $\bar{\partial}$-equation, we need the following weighted $L^2$ estimate, with the weight function not necessarily psh.

\begin{proposition}\label{prop:L2_C}
Let $\Omega\subsetneq\mathbb{C}$ be a domain such that \eqref{eq:Hardy} holds. Let $\varphi$ be a real-valued locally Lipschitz function on $\Omega$ such that
\begin{equation}\label{eq:gamma_C}
|\varphi_z|^2:=\left|\partial\varphi/\partial z\right|^2\leq\gamma^2\cdot\eta
\end{equation}
holds a.e. for some $\gamma\in(0,1)$. Then for any $v\in{L^2_{(0,1)}(\Omega,\varphi)}$ with $\int_\Omega|v|^2e^{-\varphi}/\eta<+\infty$, there exists some $u\in L^2(\Omega,\varphi)$ such that $\bar{\partial}u=v$ and
\begin{equation}\label{eq:L2_C}
\int_\Omega|u|^2e^{-\varphi}\leq\frac{4}{(1-\gamma)^2}\int_\Omega\frac{|v|^2}{\eta}e^{-\varphi}.
\end{equation}
\end{proposition}

\begin{proof}
For $\omega=\psi d\bar{z}\in\mathcal{D}_{(0,1)}(\Omega)$, we have
\[
\bar{\partial}^\ast \omega=\psi_z,\ \ \ \ \ \ \bar{\partial}^\ast_\varphi \omega=\psi_z-\varphi_z\psi.
\]
Replacing $\omega$ by $\omega e^{-\varphi/2}$ in \eqref{eq:Hardy_MKH_C}, we obtain
\begin{eqnarray*}
\int_\Omega|\omega|^2\eta e^{-\varphi}
&\leq& 4\int_\Omega\left|\frac{\partial(\psi e^{-\varphi/2})}{\partial{z}}\right|^2\\
&=& 4\int_\Omega\left|\psi_z-\frac{1}{2}\varphi_z\psi\right|^2e^{-\varphi}\\
&=& 4\int_\Omega\left|\bar{\partial}^\ast_\varphi \omega+\frac{1}{2}\varphi_z\psi\right|^2e^{-\varphi}.
\end{eqnarray*}
It follows from the Minkowski inequality that
\[
\left(\int_\Omega|\omega|^2\eta e^{-\varphi}\right)^{1/2}\leq2\left(\int_\Omega|\bar{\partial}^\ast_\varphi \omega|^2e^{-\varphi}\right)^{1/2}+\left(\int_\Omega|\omega|^2|\varphi_z|^2e^{-\varphi}\right)^{1/2},
\]
i.e.,
\begin{equation}\label{eq:Hardy_MKH_weight}
\int_\Omega|\omega|^2\eta e^{-\varphi}\leq\frac{4}{(1-\gamma)^2}\int_\Omega|\bar{\partial}^\ast_\varphi \omega|^2e^{-\varphi}
\end{equation}
holds, in view of \eqref{eq:gamma_C}. Similar to the proof of Proposition \ref{prop:L2_Cn}, let us consider the linear functional
\[
T:\bar{\partial}^\ast_\varphi \omega\mapsto (\omega,v)_\varphi,\ \ \ \omega\in \mathcal{D}_{(0,1)}(\Omega).
\]
We have
\begin{eqnarray*}
|T(\bar{\partial}^\ast_\varphi \omega)|^2 & \le & \int_\Omega|\omega|^2\eta e^{-\varphi} \cdot \int_\Omega \frac{|v|^2}{\eta}e^{-\varphi}\\
& \le & \frac{4}{(1-\gamma)^2}\int_\Omega|\bar{\partial}_\varphi^\ast \omega|^2e^{-\varphi} \cdot \int_\Omega \frac{|v|^2}{\eta}e^{-\varphi}.
\end{eqnarray*}
By the Hahn-Banach theorem and the Riesz representation theorem, there exists a solution of $\bar{\partial}u=v$ which satisfies \eqref{eq:L2_C}.
\end{proof}

Next, we use a trick of Berndtsson \cite{Berndtsson99} to verify the following

\begin{proposition}\label{prop:L2_DF_C}
If we further assume that \eqref{eq:gamma_C} holds a.e. for some $\gamma\in(0,1/3)$, then the canonical solution $u_0$ of $\bar{\partial}u=v$ in $L^2(\Omega)$ satisfies
\begin{equation}\label{eq:L2_DF_C}
\int_\Omega|u_0|^2e^\varphi\leq\frac{4}{(1-3\gamma)^2}\int_\Omega\frac{|v|^2}{\eta}e^\varphi.
\end{equation}
\end{proposition}

\begin{proof}
By a standard approximation procedure, we may assume that $\varphi$ is bounded on $\Omega$. Then $u_0e^{\varphi}$ is the canonical solution of $\bar{\partial}u=\bar{\partial}(u_0e^{\varphi})$ in $L^2(\Omega,\varphi)$, and it follows from Proposition \ref{prop:L2_C} that
\[
\int_\Omega|u_0|^2e^\varphi
\leq \frac{4}{(1-\gamma)^2}\int_\Omega\frac{|\bar{\partial}(u_0e^\varphi)|^2}{\eta}e^{-\varphi}
= \frac{4}{(1-\gamma)^2}\int_\Omega\frac{\left|v+u_0\bar{\partial}\varphi\right|^2}{\eta}e^\varphi.
\]
Since $|\varphi_z|=|\varphi_{\bar{z}}|$, we infer from the Minkowski inequality and \eqref{eq:gamma_C} that
\begin{eqnarray*}
\left(\int_\Omega|u_0|^2e^\varphi\right)^{1/2}
&\leq& \frac{2}{1-\gamma}\left(\int_\Omega\frac{|v|^2}{\eta}e^\varphi\right)^{1/2} + \frac{2}{1-\gamma}\left(\int_\Omega|u_0|^2\frac{|\varphi_{\bar{z}}|^2}{\eta}e^\varphi\right)^{1/2}\\
&\leq& \frac{2}{1-\gamma}\left(\int_\Omega\frac{|v|^2}{\eta}e^\varphi\right)^{1/2} + \frac{2\gamma}{1-\gamma}\left(\int_\Omega|u_0|^2e^\varphi\right)^{1/2},
\end{eqnarray*}
from which \eqref{eq:L2_DF_C} immediately follows.
\end{proof}

A simple consequence of Proposition \ref{prop:L2_DF_C} is the following

\begin{corollary}\label{cor:L2_canonical_solution}
Let $0<\alpha<2h(\Omega)/3$. Then the canonical solution $u_0$ of $\bar{\partial}u=v$ satisfies
\begin{equation}\label{eq:L2_minimal_1}
\int_{\Omega}|u_0|^2\delta_\Omega^\alpha\leq C\int_{\Omega}|v|^2\delta_\Omega^{2+\alpha}
\end{equation}
with
\[
C=\frac{16h(\Omega)^2}{(2h(\Omega)-3\alpha)^2}.
\]
\end{corollary}

\begin{proof}
Suppose $h(\Omega)>0$, i.e., the Hardy type inequality
\[
c^2\int_\Omega\frac{\phi^2}{\delta_\Omega^2}\leq\int_\Omega|\nabla\phi|^2
\]
holds for all $c\in(0,h(\Omega))$. We take $\eta=c^2/\delta_\Omega^2$ and $\varphi=\alpha\log\delta_\Omega$. Note that
\[
\left|\varphi_z\right|^2=\frac{1}{4}|\nabla\varphi|^2=\frac{\alpha^2}{4}\frac{|\nabla\delta_\Omega|^2}{\delta_\Omega^2}=\frac{\alpha^2}{4\delta_\Omega^2}=\frac{\alpha^2}{4c^2}\eta
\]
holds a.e. on $\Omega$. Since
\[
\gamma:=\frac{\alpha}{2c}\in(0,1/3)
\]
when $c$ is sufficiently close to $h(\Omega)$, so Proposition \ref{prop:L2_DF_C} applies.
\end{proof}

\section{Proofs of Theorem \ref{th:Hartogs_L2} and Theorem \ref{th:Bergman_int}}

\begin{proof}[Proof of Theorem \ref{th:Hartogs_L2}]
Set $d_E(z):=d(z,E)$ and $d_H(z):=d(z,H)$.  Choose a smooth function $\chi:\mathbb R\rightarrow [0,1]$ such that $\chi|_{(-\infty,1/2]}=0$ and $\chi|_{[1,\infty)}=1$.  Given $f\in A^2(\Omega\setminus E)$,  define
\[
v:=\begin{cases}
\bar{\partial}(\chi(d_E/r)f),\ \ \ &\text{on }\Omega\setminus E,\\
0,\ \ \ &\text{on }E.
\end{cases}
\]
Clearly,  $v$ is a smooth $\bar{\partial}-$closed $(0,1)-$form on $\mathbb C^n$.  Moreover,  since $|\nabla d_E|\le 1$ a.e.,  we have
$$
\int_{\mathbb C^n} |v|^2\le \frac{\sup|\chi'|^2}{r^2}\cdot \int_{r/2\le d_E\le r} |f|^2\le \frac{\sup|\chi'|^2}{r^2}\cdot \int_{\Omega\backslash E} |f|^2<\infty,
$$
and since $E\subset H_R$,  it follows that
\begin{eqnarray*}
\int_{\mathbb C^n} |v|^2 d_H^2 & \le & \frac{\sup|\chi'|^2}{r^2}\cdot \int_{r/2\le d_E\le r} |f|^2d_H^2\\
& \le &  \frac{(R+r)^2}{r^2}\cdot \sup|\chi'|^2\cdot  \int_{\Omega\backslash E} |f|^2<\infty.
\end{eqnarray*}
Thanks to Proposition \ref{prop:L2_Cn},  we obtain a solution of $\bar{\partial}u=v$ which satisfies
$$
\int_{\mathbb C^n} |u|^2 \lesssim \int_{\Omega\backslash E} |f|^2.
$$
Since $\mathrm{supp\,}v\subset E_r$,  we conclude that $u\in \mathcal O(\mathbb C^n\backslash E_r)$.  Also,  since $H_{R+r}\subsetneq \mathbb C^n$ is convex,  so there exists a real hyperplane $\mathcal H$ in $\mathbb C^n$ such that $d(\mathcal H,H_{R+r})>1$.  As $n\ge 2$, $\mathcal H$ contains at least one complex line $\mathcal{L}$.  Without loss of generality,  we assume  $\mathcal{L}=\{z'=0\}$ where $z'=(z_1,\cdots,z_{n-1})$.  Thus the cylinder
\[
\mathcal C:=\mathcal{L}\times \mathbb B^{n-1}\subset \mathbb C^n\backslash H_{R+r}\subset \mathbb C^n\backslash E_r,
\]
where $\mathbb B^{n-1}$ is the unit ball in $\mathbb C^{n-1}$.  Now $u\in A^2(\mathbb C^n)$,  so $u(z',\cdot)\in A^2(\mathbb C)$ for every $z'\in \mathbb B^{n-1}$,  and it has to vanish in view of the ($L^2$) Liouville theorem.    By the theorem of unique continuation,  $u=0$ in an unbounded component of $\mathbb C^n\backslash E_r$,  which naturally intersects with $\Omega\backslash E$.    Finally,  the function $F:=\chi(d_E/r)f-u$ is holomorphic on $\Omega$ and satisfies $F=f$ on a nonempty open subset in $\Omega\backslash E$.  Since $\Omega\backslash E$ is connected,  it follows that $F=f$ on $\Omega\backslash E$.  Clearly,
\[
\int_\Omega |F|^2 \le 2\int_\Omega |\chi(d_E/r)f|^2 + 2\int_\Omega |u|^2\le \int_{\Omega\backslash E} |f|^2 + 2\int_{\mathbb C^n} |u|^2<\infty.\qedhere
\]
\end{proof}

\begin{remark}
Thanks to the Closed Graph Theorem, we see that the extension operator
\[
I_E:A^2(\Omega\setminus{E})\rightarrow A^2(\Omega)
\]
is bounded, i.e., there exists a constant $C_E$ such that
\begin{equation}\label{eq:Hartogs_estimate}
\int_{\Omega}|I_E(f)|^2\leq C_E\int_{\Omega\setminus{E}}|f|^2,\ \ \ \forall\,f\in{A^2(\Omega\setminus E)}.
\end{equation}
\end{remark}

To prove Theorem \ref{th:Bergman_int}, we need the following

\begin{proposition}\label{prop:Bergman_decaying}
Let $\Omega\subsetneq\mathbb{C}$ be a domain with $h(\Omega)>0$. For any $w\in\Omega$ and $0<c<h(\Omega)$, there exists some constant $C=C_{\Omega,w,c}>0$ such that
\begin{equation}\label{eq:boundary_decay}
\int_{\delta_\Omega\leq\varepsilon}|K_{\Omega}(\cdot,w)|^2\leq C \varepsilon^{2c/3},\ \ \ \forall\,\varepsilon>0.
\end{equation}
\end{proposition}

\begin{proof}
It suffices to consider the case $0<\varepsilon\ll1$. If $c<c'<h(\Omega)$, then \eqref{eq:L2_minimal_1} holds for $\alpha=2c'/3$. Let $\chi:\mathbb{R}\rightarrow[0,1]$ be a Lipschitz function with $\chi|_{(-\infty,1]}=1$, $\chi|_{[3/2,+\infty)}=0$ and $|\chi'|\leq2$. Given $w\in\Omega$, set
\[
v:=K_{\Omega}(\cdot,w)\bar{\partial}\chi(\delta_\Omega/\varepsilon).
\]
Then
\begin{equation}\label{eq:Bergman_proj}
\int_{\Omega}\chi(\delta_\Omega/\varepsilon)K_{\Omega}(\cdot,w)\overline{K_{\Omega}(\cdot,\zeta)}=\chi(\delta_\Omega(\zeta)/\varepsilon)K_{\Omega}(\zeta,w)-u_0(\zeta),\ \ \ \forall\,\zeta\in\Omega.
\end{equation}
Here $u_0$ is the canonical solution in Corollary \ref{cor:L2_canonical_solution}, which satisfies
\begin{eqnarray}\label{eq:L2_Bergman_int}
\int_{\Omega}|u_0|^2\delta_\Omega^\alpha
&\leq& \frac{C_\alpha}{\varepsilon^2}\int_{\varepsilon\leq\delta_\Omega\leq3\varepsilon/2}|K_{\Omega}(\cdot,w)|^2\delta_\Omega^{\alpha+2}\nonumber\\
&\leq& C_\alpha\varepsilon^\alpha \int_{\varepsilon\leq\delta_\Omega\leq3\varepsilon/2}|K_{\Omega}(\cdot,w)|^2
\end{eqnarray}
for some generic constant $C_\alpha$ depending on $\alpha$. Since $\chi(\delta_\Omega(w)/\varepsilon)=0$ when $0<\varepsilon\ll1$, it follows from \eqref{eq:Bergman_proj} that
\begin{equation}\label{eq:Bergman_proj_2}
\int_{\delta_\Omega\leq\varepsilon}|K_{\Omega}(\cdot,w)|^2\leq-u_0(w).
\end{equation}
Since $u_0$ is holomorphic when $\delta_\Omega>3\varepsilon/2$, we infer from the mean-value inequality and \eqref{eq:L2_Bergman_int} that
\begin{eqnarray*}
|u_0(w)|^2
&\leq& \frac{1}{\pi\delta_\Omega(w)^2}\int_{\Delta(w,\delta_\Omega(w)/2)}|u_0|^2\\
&\leq& \frac{1}{\pi\delta_\Omega(w)^2}\left[\sup_{\Delta(w,\delta_\Omega(w)/2)}\delta_\Omega^{-\alpha}\right]\int_{\Delta(w,\delta_\Omega(w)/2)}|u_0|^2\delta_\Omega^\alpha\\
&\leq& {C_{\Omega,\alpha,w}}\int_{\Omega}|u_0|^2\delta_\Omega^\alpha\\
&\leq& {C_{\Omega,\alpha,w}}\,\varepsilon^{\,\alpha}\int_{\delta_\Omega\leq3\varepsilon/2}|K_{\Omega}(\cdot,w)|^2.
\end{eqnarray*}
This combined with \eqref{eq:Bergman_proj_2} yields
\[
\int_{\delta_\Omega\leq\varepsilon}|K_{\Omega}(\cdot,w)|^2\leq{C_{\Omega,\alpha,w}}\,\varepsilon^{\alpha/2}\left(\int_{\delta_\Omega\leq3\varepsilon/2}|K_\Omega(\cdot,w)|^2\right)^{1/2}.
\]
By iteration, we conclude that
\[
\int_{\delta_\Omega\leq\varepsilon}|K_\Omega(\cdot,w)|^2\leq{C_{\Omega,\alpha,w,k}}\,\varepsilon^{\alpha/2+\alpha/4+\cdots+\alpha/2^k},
\]
from which \eqref{eq:boundary_decay} immediately follows, since the exponent can be arbitrarily close to $\alpha$.
\end{proof}

\begin{proof}[Proof of Theorem \ref{th:Bergman_int}]

Fix $z\in\Omega$ with $0<\delta_\Omega(z)\ll1$ and take $\varepsilon=2\delta_\Omega(z)$ in \eqref{eq:boundary_decay}. Applying the mean value inequality on the disc $\Delta(z,\delta_\Omega(z))\subset\{\delta_\Omega\leq\varepsilon\}$, we have
\[
|K_\Omega(z,w)|^2
\leq \frac{1}{\pi\delta_\Omega(z)^2}\int_{\delta_\Omega\leq\varepsilon}|K_\Omega(\cdot,w)|^2
\leq C_{\Omega,w,c}\delta_\Omega(z)^{2c/3-2},
\]
for any $0<c<h(\Omega)$. Fix $k_0\gg1$. We have
\begin{eqnarray*}
\int_\Omega|K_\Omega(\cdot,w)|^{2+\beta}
&\lesssim& 1+\sum_{k\geq{k_0}}\int_{2^{-k-1}\leq\delta_\Omega<2^{-k}}|K_\Omega(\cdot,w)|^{2+\beta}\\
&\lesssim& 1+\sum_k2^{(k+1)(1-c/3)\beta}\int_{\delta_\Omega<2^{-k}}|K_\Omega(\cdot,w)|^2\\
&\lesssim& 1+\sum_k2^{(k+1)(1-c/3)\beta-2ck/3}<+\infty,
\end{eqnarray*}
provided $(1-c/3)\beta<2c/3$. Thus
\[
\beta(\Omega)\geq2+\frac{2c/3}{1-c/3}=\frac{6}{3-c}.
\]
Letting $c\rightarrow{h(\Omega)}$, we conclude the proof.
\end{proof}

\section{Sobolev-Morrey inequality and $L^p$ estimates for the $\bar{\partial}$-equation}

Our proof of Theorem \ref{th:absolute_continuity} relies on the following $L^p$ estimate for the $\bar{\partial}$-equation due to Berndtsson \cite{Berndtsson92}.

\begin{theorem}\label{th:Berndtsson_Lp}
Let $\Omega\subsetneq\mathbb{C}$ be a domain and $1\leq p<2$. Then for any $v\in{L^1_{(0,1)}(\Omega)}$, there exists some $u\in{L^p(\Omega)}$ such that $\bar{\partial}u=v$ and
\begin{equation}\label{eq:Berndtsson_Lp}
\left(\int_\Omega|u|^p\right)^{1/p}\leq C_{p,\Omega}\int_\Omega|v|.
\end{equation}
\end{theorem}

Since explicit estimates of the constant $C_{p,\Omega}$ as $p\rightarrow2-$ will be used in the proof of Theorem \ref{th:absolute_continuity_effective},  we decide to include here a very simple proof for $p\rightarrow2-$, based on the Sobolev-Morrey inequality \eqref{eq:Morrey}.  

\begin{proof}
Take $q$ with $1/p+1/q=1$, so that $q\rightarrow2+$ as $p\rightarrow2-$. As before, we consider the linear functional
\[
T:\bar{\partial}^*\omega\mapsto(\omega,v),\ \ \ \omega\in{\mathcal{D}_{(0,1)}(\Omega)}.
\]
For $\omega=\phi d\bar{z}$, we have
\begin{eqnarray*}
|T(\bar{\partial}^\ast\omega)|
&\leq& |(\omega,v)| \leq \left(\int_\Omega|v|\right)\sup_\Omega|\phi|\\
&\lesssim& (q-2)^{-1/2}|\Omega|^{1/2-1/q}\left(\int_\Omega|\nabla\phi|^q\right)^{1/q}\left(\int_\Omega|v|\right),
\end{eqnarray*}
in view of \eqref{eq:Morrey}. Here and in what follows, the implicit constants are numerical. Recall that
\[
\left(\int_\Omega|\bar\partial\phi|^q\right)^{1/q}
\lesssim (q-1)^q\left(\int_\Omega|\partial\phi|^q\right)^{1/q}
\lesssim \left(\int_\Omega|\partial\phi|^q\right)^{1/q}
\]
(see \cite{AhlforsBook}, Chapter V. D.). Thus
\[
\left(\int_\Omega|\nabla\phi|^q\right)^{1/q}
\lesssim \left(\int_\Omega|\partial\phi|^q\right)^{1/q}
\lesssim \left(\int_\Omega|\bar\partial^\ast \omega|^q\right)^{1/q},
\]
so that
\[
|T(\bar{\partial}^\ast \omega)|\lesssim (q-2)^{-1/2}|\Omega|^{1/2-1/q}\left(\int_\Omega|\bar{\partial}^\ast \omega|^q\right)^{1/q}\left(\int_\Omega|v|\right).
\]
Since $(L^q(\Omega))^*=L^p(\Omega)$, $\bar{\partial}u=v$ admits a solution $u\in L^p(\Omega)$ ($1/p+1/q=1$) with
\begin{eqnarray*}
\left(\int_\Omega|u|^p\right)^{1/p}
&\lesssim& (q-2)^{-1/2}|\Omega|^{1/2-1/q}\int_\Omega|v|\\
&\lesssim& (2-p)^{-1/2}|\Omega|^{1/p-1/2}\int_\Omega|v|.
\end{eqnarray*}
Namely, \eqref{eq:Berndtsson_Lp} holds with $C_{p,\Omega}=C_0(2-p)^{-1/2}|\Omega|^{1/p-1/2}$ for some numerical constant $C_0$.
\end{proof}

\section{Proofs of Theorem \ref{th:absolute_continuity}, Theorem \ref{th:absolute_continuity_effective} and Proposition \ref{prop:Hartogs_absolute_continuity}}

Let us first review some basic facts about the potential theory  in the plane. Let $E\subset\mathbb{C}$ be a compact set and $U$ a neighbourhood of $E$. For any Borel probability measure $\mu$ supported in $E$, the Green potential and Green energy of $\mu$ relative to $U$ are given by
\[
p_\mu(z)=\int_{U} g_{U}(z,w)d\mu(w),\ \ \ z\in U
\]
and
\[
I(\mu)=\int_{U} p_\mu d\mu =\int_{U}\int_{U} g_{U}(z,w) d\mu(z)d\mu(w).
\]
By Frostman's theorem, there is an equilibrium measure $\mu_0$, which maximizes $I(\mu)$ among all Borel probability measures on $E$, such that
\begin{itemize}
\item[$(1)$]
$p_{\mu_0}\geq I(\mu_0)$ on $U$,

\item[$(2)$]
$p_{\mu_0}=I(\mu_0)$ on the interior of $E$ and $p_{\mu_0}=I(\mu_0)$ n.e. on $\partial E$.

\item[$(3)$]
$p_{\mu_0}=0$ on $\partial U$.

\item[$(4)$]
$p_{\mu_0}$ is harmonic on $U\setminus E$.
\end{itemize}
Moreover, we have $\Delta p_{\mu_0}=2\pi\mu_0$, so that
\[
I(\mu_0)=\frac{1}{2\pi}\int_{U}p_{\mu_0}\Delta p_{\mu_0}=-\frac{1}{2\pi}\int_U|\nabla p_{\mu_0}|^2.
\]
The Green capacity of $E$ relative to $U$ is given by
\[
\mathcal{C}_g(E,U):=e^{I(\mu_0)}.
\]
Recall that the logarithmic capacity $\mathcal{C}_l(E)$ is defined in a similar way by replacing $g_U(z,w)$ with $\log|z-w|$. Since $g_U(z,w)\leq\log|z-w|-\log d(E,\partial U)$, we have
\begin{equation}\label{eq:Green_log_cap}
\mathcal{C}_g(E,U)\leq\frac{\mathcal{C}_l(E)}{d(E,\partial U)}.
\end{equation}

\begin{proof}[Proof of Theorem \ref{th:absolute_continuity} (The case $z=w$)]
The idea is to approximate $K_{\Omega\setminus{E}}(\cdot,z)$ by functions in $A^2(\Omega)$. By Choquet's theorem (cf. Ransford's book \cite{Ransford}, (5.3)), there is a sequence $\{U_j\}$ of open neighbourhoods of $E$ with
\begin{itemize}
\item[$(1)$]
$U_{j+1}\subset\subset{U_j}$,

\item[$(2)$]
$\lim_{j\rightarrow\infty}\delta_{\Omega\setminus{\overline{U}_j}}(z)=\delta_{\Omega\setminus{E}}(z)$,

\item[$(3)$]
$\lim_{j\rightarrow\infty}\mathcal{C}_l(\overline{U}_j)=\mathcal{C}_l(E)$ as $j\rightarrow\infty$.
\end{itemize}
Set $D_j:=\Delta(z,\delta_{\Omega\setminus{\overline{U}_j}}(z)/2)$. Let us consider the equilibrium measure $\mu_j$ of $\overline{U}_j$ with respect to $\mathbb{C}\setminus\overline{D}_j$ and take
\[
\chi_j:=\frac{p_{\mu_j}}{I(\mu_j)}.
\]
Then $\chi_j|_{\overline{U}_j}=1$, $\chi_j|_{\overline{D}_j}=0$ and
\[
\int_{\mathbb{C}}|\nabla\chi_j|^2=\frac{1}{I(\mu_j)^2}\int_{\mathbb{C}}|\nabla p_{\mu_j}|^2=-\frac{2\pi}{I(\mu_j)}=:c_j.
\]
Since $d(\overline{U}_j,\partial{D_j})=\delta_{\Omega\setminus{\overline{U}_j}}(z)/2$, it follows from \eqref{eq:Green_log_cap} that
\[
c_j
= -\frac{2\pi}{\log\mathcal{C}_g(\overline{U}_j,\mathbb{C}\setminus\overline{D}_j)}
\leq 2\pi\left(\log\frac{\delta_{\Omega\setminus{\overline{U}_j}}(z)}{2\mathcal{C}_l(\overline{U}_j)}\right)^{-1}
\]
when $\mathcal{C}_l(\overline{U}_j)\leq \delta_{\Omega\setminus{\overline{U}_j}}(z)/2$, so that
\begin{equation}\label{eq:limsup_capacity}
\limsup_{j\rightarrow+\infty} c_j
\leq 2\pi\left(\log\frac{r_0}{2\mathcal{C}_l(E)}\right)^{-1}
\end{equation}
provided $\mathcal{C}_l(E)<r_0/2$. 

Set
\[
v_j:=\bar{\partial}((1-\chi_j)K_{\Omega\setminus{E}}(\cdot,z))=-K_{\Omega\setminus{E}}(\cdot,z)\bar{\partial}\chi_j.
\]
Let $3/2<p<2$. In what follows, $C_{p,\Omega}$ denotes a constant of the form
\[
C_{p,\Omega}=C_0(2-p)^{-1/2}|\Omega|^{1/p-1/2},
\]
where $C_0$ is a generic numerical constant. By Theorem \ref{th:Berndtsson_Lp}, there exists a solution $u_j\in{L^p(\Omega)}$ of $\bar{\partial}u_j=v_j$ such that
\begin{eqnarray*}
\left(\int_\Omega|u_j|^p\right)^{1/p}
&\leq& C_{p,\Omega}\int_\Omega|v_j| \leq C_{p,\Omega}\int_{\Omega\setminus{E}}|K_{\Omega\setminus{E}}(\cdot,z)||\bar{\partial}\chi_j|\\
&\leq& C_{p,\Omega}\left(\int_{\Omega\setminus{E}}|K_{\Omega\setminus{E}}(\cdot,z)|^2\right)^{1/2}\left(\int_{\Omega\setminus{E}}|\bar{\partial}\chi_j|^2\right)^{1/2}\\
&\leq& C_{p,\Omega} K_{\Omega\setminus{E}}(z)^{1/2} c_j^{1/2}.
\end{eqnarray*}
Note that $u_j$ is holomorphic in $D_j=\Delta(z,\delta_{\Omega\setminus{\overline{U}_j}}(z)/2)$. Thus we get
\begin{eqnarray*}
|u_j(z)|
&\leq&\frac{1}{|\Delta(z,\delta_{\Omega\setminus{\overline{U}_j}}(z)/2)|^{1/p}}\left(\int_{\Delta(z,\delta_{\Omega\setminus{\overline{U}_j}}(z)/2)}|u_j|^p\right)^{1/p}\\
&\leq& C_{p,\Omega} K_{\Omega\setminus E}(z)^{1/2}\delta_{\Omega\setminus{E}}(z)^{-2/p} c_j^{1/2}\\
&\leq& C_{p,\Omega} K_{\Omega\setminus E}(z)^{1/2}r_0^{-2/p} c_j^{1/2}
\end{eqnarray*}
for $j\gg1$. If we define $f_j=(1-\chi_j)K_{\Omega\setminus{E}}(\cdot,z)-u_j$, then $f_j$ is holomorphic on $\Omega$, such that
\begin{eqnarray*}
|f_j(z)|
&\geq& K_{\Omega\setminus{E}}(z)-|u_j(z)|\\
&\geq& K_{\Omega\setminus{E}}(z)-C_{p,\Omega} K_{\Omega\setminus E}(z)^{1/2}r_0^{-2/p} c_j^{1/2}
\end{eqnarray*}
and
\begin{eqnarray*}
\left(\int_\Omega|f_j|^p\right)^{1/p}
&\leq& \left(\int_\Omega|(1-\chi_j)K_{\Omega\setminus{E}}(\cdot,z)|^p\right)^{1/p}+\left(\int_\Omega|u_j|^p\right)^{1/p}\\
&\leq& \left(\int_{\Omega\setminus E}|K_{\Omega\setminus{E}}(\cdot,z)|^p\right)^{1/p}+\left(\int_\Omega|u_j|^p\right)^{1/p}\\
&\leq& |\Omega|^{1/p-1/2}K_{\Omega\setminus E}(z)^{1/2}+C_{p,\Omega} K_{\Omega\setminus{E}}(z)^{1/2} c_j^{1/2},
\end{eqnarray*}
in view of H\"{o}lder's inequality. Recall that the $p$-Bergman kernel is defined by
\[
K_{\Omega,p}(z):=\sup\left\{|f(z)|^p:f\in\mathcal{O}(\Omega),\ \int_\Omega|f|^p\leq1\right\}.
\]
Thus
\[
K_{\Omega,p}(z)^{1/p}
\geq \frac{|f_j(z)|}{(\int_\Omega|f_j|^p)^{1/p}}
\geq \frac{|\Omega|^{1/2-1/p}K_{\Omega\setminus{E}}(z)^{1/2}-C_0(2-p)^{-1/2}r_0^{-2/p}c_j^{1/2}}{1+C_0(2-p)^{-1/2}c_j^{1/2}},
\]
i.e.,
\begin{eqnarray}\label{eq:difference1}
& & K_{\Omega\setminus{E}}(z)^{1/2}-K_{\Omega,p}(z)^{1/p}\nonumber\\
&\leq& \left(1-|\Omega|^{1/2-1/p}\right)K_{\Omega\setminus{E}}(z)^{1/2}+C_0(2-p)^{-1/2}\left(r_0^{-2/p}+K_{\Omega,p}(z)^{1/p}\right)c_j^{1/2}.
\end{eqnarray}
By the mean-value inequality, we have
\[
K_{\Omega\setminus{E}}(z)^{1/2}\leq\frac{1}{\pi^{1/2}\delta_{\Omega\setminus{E}}(z)}\leq\frac{1}{\pi^{1/2}r_0}
\]
and
\[
K_{\Omega,p}(z)^{1/p}\leq\frac{1}{\pi^{1/p}\delta_\Omega(z)^{2/p}}\leq\frac{1}{\pi^{1/p}r_0^{2/p}}.
\]
Then \eqref{eq:difference1} yields
\[
K_{\Omega\setminus{E}}(z)^{1/2}-K_{\Omega,p}(z)^{1/p} \leq C_1(2-p)+C_2(2-p)^{-1/2}c_j^{1/2}
\]
for some constants $C_1$ and $C_2$ depending on $r_0$ and $\Omega$. On the other hand, it follows from Proposition 6.1 in \cite{CZ22} that
\[
\lim_{p\rightarrow2-}K_{\Omega,p}(z)^{1/p}=K_\Omega(z)^{1/2},\ \ \ \forall\,z\in\Omega.
\]
Since $|\Omega|^{1/p}K_{\Omega,p}(z)^{1/p}$ is decreasing in $p$ for each fixed $z$ (cf. (6.3) in \cite{CZ22}), we infer from Dini's theorem that the above convergence is locally uniform. Thus for any $z\in\Omega$ with $r_0\leq\delta_{\Omega\setminus E}(z)\leq\delta_\Omega(z)$, we have
\[
K_{\Omega,p}(z)^{1/p}-K_\Omega(z)^{1/2}\leq \varepsilon_p
\]
for some constant $\varepsilon_p$ with  $\lim_{p\rightarrow2-}\varepsilon_p=0$. Thus
\[
K_{\Omega\setminus{E}}(z)^{1/2}-K_{\Omega}(z)^{1/2}\leq C_1(2-p)+\varepsilon_p+C_2(2-p)^{-1/2}c_j^{1/2}.
\]
Letting $j\rightarrow+\infty$, we have
\begin{equation}\label{eq:difference2}
K_{\Omega\setminus{E}}(z)^{1/2}-K_{\Omega}(z)^{1/2}\leq C_1(2-p)+\varepsilon_p+C_2(2-p)^{-1/2}\left(\log\frac{r_0}{2\mathcal{C}_l(E)}\right)^{-1/2}
\end{equation}
in view of \eqref{eq:limsup_capacity}, from which the assertion immediately follows, for the last term converges to zero as $\mathcal{C}_l(E)\rightarrow0+$ for any fixed $p\in(3/2,2)$.
\end{proof}

\begin{remark}
It seems that the Sobolev inequality
\begin{equation}\label{eq:Sobolev}
\left(\int_\Omega|\phi|^{\frac{2p}{2-p}}\right)^{\frac{2-p}{2p}}\leq C_p\left(\int_\Omega|\nabla \phi|^p\right)^{1/p},\ \ \ \forall\,\phi\in{C^\infty_0(\Omega)},\ 1\leq p<2
\end{equation}
does not work in the above argument. Thus we have to use a limit case of \eqref{eq:Sobolev}, i.e., the Sobolev-Morrey inequality \eqref{eq:Morrey}.
\end{remark}

The general case is a direct consequence of the following result,  which is communicated to us by Xu Wang at NTNU, Trondheim.

\begin{proposition}\label{prop:Bergman_monotonic}
Let $\Omega\subset\mathbb{C}^n$ be a domain. Define
\[
R_\Omega^{\pm}(z,w):=K_\Omega(z)+K_\Omega(w)\pm2\mathrm{Re}\,K_\Omega(z,w)
\]
and
\[
I_\Omega^{\pm}(z,w):=K_\Omega(z)+K_\Omega(w)\pm2\mathrm{Im}\,K_\Omega(z,w).
\]
Then
\[
R_\Omega^{\pm}(z,w)=\max\left\{|f(z)\pm f(w)|^2:f\in{A^2(\Omega)},\ \|f\|_{L^2(\Omega)}=1\right\},
\]
\[
I_\Omega^{\pm}(z,w)=\max\left\{|f(z)\mp \mathrm{i}f(w)|^2:f\in{A^2(\Omega)},\ \|f\|_{L^2(\Omega)}=1\right\}.
\]
In particular, if\/ $\Omega_1\subset\Omega_2\subset\mathbb{C}^n$ are domains, then
\begin{equation}\label{eq:Bergman_monotonic_0}
R_{\Omega_1}^{\pm}(z,w)\geq R_{\Omega_2}^{\pm}(z,w),\ \ \ I_{\Omega_1}^{\pm}(z,w)\geq I_{\Omega_2}^{\pm}(z,w).
\end{equation}
\end{proposition}

\begin{proof}
We shall only prove the proposition for $R_\Omega^{\pm}$. By the reproducing formula, we have
\[
f(z)\pm f(w)=\int_\Omega f\cdot\overline{K_\Omega(\cdot,z)\pm K_\Omega(\cdot,w)},\ \ \ \forall\,f\in{A^2(\Omega)}.
\]
This combined with the Cauchy-Schwarz inequality yields
\begin{equation}\label{eq:Bergman_monotonic}
|f(z)\pm f(w)|^2\leq\|K_\Omega(\cdot,z)\pm K_\Omega(\cdot,w)\|_{L^2(\Omega)},
\end{equation}
provided $\|f\|_{L^2(\Omega)}=1$. Equality in \eqref{eq:Bergman_monotonic} holds if
\[
f=\frac{K_\Omega(\cdot,z)\pm K_\Omega(\cdot,w)}{\|K_\Omega(\cdot,z)\pm K_\Omega(\cdot,w)\|_{L^2(\Omega)}},
\]
so that
\[
\|K_\Omega(\cdot,z)\pm K_\Omega(\cdot,w)\|_{L^2(\Omega)}=\max\left\{|f(z)\pm f(w)|^2:f\in{A^2(\Omega)},\ \|f\|_{L^2(\Omega)}=1\right\}.
\]
Finally, a straightforward calculation shows $\|K_\Omega(\cdot,z)\pm K_\Omega(\cdot,w)\|_{L^2(\Omega)}=R^{\pm}_\Omega(z,w)$.
\end{proof}

\begin{proof}[Proof of Theorem \ref{th:absolute_continuity} (The general case)]
Take $\Omega_1=\Omega\setminus E$ and $\Omega_2=\Omega$ in \eqref{eq:Bergman_monotonic_0}, we have
\begin{eqnarray*}
2\left|\mathrm{Re}\left(K_{\Omega\setminus{E}}(z,w)-K_\Omega(z,w)\right)\right|
&\leq& \left(K_{\Omega\setminus{E}}(z)-K_\Omega(z)\right)+\left(K_{\Omega\setminus{E}}(w)-K_\Omega(w)\right),\\
2\left|\mathrm{Im}\left(K_{\Omega\setminus{E}}(z,w)-K_\Omega(z,w)\right)\right|
&\leq& \left(K_{\Omega\setminus{E}}(z)-K_\Omega(z)\right)+\left(K_{\Omega\setminus{E}}(w)-K_\Omega(w)\right).
\end{eqnarray*}
Thus
\[
\left|K_{\Omega\setminus{E}}(z,w)-K_\Omega(z,w)\right|
\leq \left(K_{\Omega\setminus{E}}(z)-K_\Omega(z)\right)+\left(K_{\Omega\setminus{E}}(w)-K_\Omega(w)\right).
\]
We are reduced to the special case proved earlier.
\end{proof}

\begin{proof}[Proof of Theorem \ref{th:absolute_continuity_effective}]
In \cite{CX24}, we showed that for any compact set $S\subset \Omega$,
\[
\left|K_{\Omega,p}(z)^{1/p}-K_\Omega(z)^{1/2}\right|\leq C_{\Omega,S} (2-p)\log\frac{1}{2-p},\ \ \ \forall\,z\in S
\]
when $0<2-p\ll1$. In particular, we have
\[
\left|K_{\Omega,p}(z)^{1/p}-K_\Omega(z)^{1/2}\right|\leq C_{\alpha,r_0,\Omega}(2-p)^{\frac{\alpha}{1-2\alpha}}
\]
for any $0<\alpha<1/3$. This cimbined with \eqref{eq:difference2} yields
\[
K_{\Omega\setminus{E}}(z)^{1/2}-K_{\Omega}(z)^{1/2}\leq C_{\alpha,r_0,\Omega}(2-p)^{\frac{\alpha}{1-2\alpha}}+C_2(2-p)^{-1/2}\left[\log\frac{r_0}{2\mathcal{C}_l(E)}\right]^{-1/2}.
\]
The assertion follows if we take $2-p\asymp (\log\frac{r_0}{2\mathcal{C}_l(E)})^{2\alpha-1}$.
\end{proof}

\begin{proof}[Proof of Proposition \ref{prop:Hartogs_absolute_continuity}]
Again it suffices to consider the case $z=w$.  Thanks to Theorem \ref{th:Hartogs_L2}, $K_{\Omega\setminus{E}}(\cdot,z)$ can be extended to an $L^2$ holomorphic function $F$ on $\Omega$. Thus for any $z\in\Omega\setminus{E}$,  we have
\begin{equation}\label{eq:difference_volume}
K_\Omega(z)\geq\frac{|F(z)|^2}{\int_\Omega|F|^2}=\frac{K_{\Omega\setminus E}(z)^2}{K_{\Omega\setminus E}(z)+\int_E|F|^2}.
\end{equation}
By the first condition of Theorem \ref{th:Hartogs_L2}, we obtain $B_r(\zeta)\subset\Omega$ for any $\zeta\in E$ and some $r>0$, so that
\[
|F(\zeta)|^2\leq\frac{1}{\omega_n r^{2n}}\int_\Omega|F|^2.
\]
Here $\omega_n=\pi^n/n!$ is the volume of the unit ball in $\mathbb{C}^n$. This combined with \eqref{eq:Hartogs_estimate} yields
\[
|F(\zeta)|^2\leq\frac{C_{E_0}}{\omega_n r^{2n}}\int_{\Omega\setminus E}|K_{\Omega\setminus E}(\cdot,z)|^2=\frac{C_{E_0}}{\omega_n r^{2n}}K_{\Omega\setminus E}(z),\ \ \ \forall\,\zeta\in{E}.
\]
Thus \eqref{eq:difference_volume} gives
\[
K_\Omega(z)\geq\frac{K_{\Omega\setminus E}(z)}{1+C_{E_0}\omega_n^{-1}r^{-2n}|E|},
\]
i.e.,
\[
K_{\Omega\setminus E}(z)-K_\Omega(z)\leq\frac{C_{E_0}}{\omega_n r^{2n}}|E|K_\Omega(z),
\]
from which the assertion follows immediately.
\end{proof}

Conversely, we have

\begin{proposition}\label{prop:capacity_zero}
Let $\Omega\subset\mathbb{C}$ be a domain and $\{E_j\}$ a sequence of closed subsets in $\Omega$ with $E_{j+1}\subset{E_j}$ and $\bigcap^\infty_{j=1}E_j=E$. If $\lim_{j\rightarrow+\infty}K_{\Omega\setminus E_j}(z)=K_\Omega(z)$ for any $z\in\Omega\setminus E$, then $E$ is a polar set, i.e., $\mathcal{C}_l(E)=0$.
\end{proposition}

\begin{proposition}\label{prop:volume_zero}
Let $\Omega\subset\mathbb{C}^n$ ($n\geq2$) be a domain and $\{E_j\}$ a sequence of closed subsets in $\Omega$ with $E_{j+1}\subset{E_j}$ and $\bigcap^\infty_{j=1}E_j=E$. If $\lim_{j\rightarrow+\infty}K_{\Omega\setminus E_j}(z)=K_\Omega(z)\neq0$ for some $z\in\Omega\setminus E$, then $|E|=0$.
\end{proposition}

\begin{proof}[Proof of Proposition \ref{prop:capacity_zero}]
It follows from Ramadanov's theorem that
\[
K_{\Omega\setminus{E}}(z)=\lim_{j\rightarrow+\infty}K_{\Omega\setminus E_j}(z)=K_\Omega(z).
\]
By the reproducing formula, we have
\begin{eqnarray*}
\int_{\Omega\setminus E}|K_{\Omega\setminus E}(\cdot,z)-K_\Omega(\cdot,z)|^2
&=& K_{\Omega\setminus E}(z)-2K_\Omega(z)+\int_{\Omega\setminus E}|K_\Omega(\cdot,z)|^2\\
&\leq& K_{\Omega\setminus E}(z)-2K_\Omega(z)+\int_\Omega|K_\Omega(\cdot,z)|^2\\
&=& K_{\Omega\setminus E}(z)-K_\Omega(z)\\
&=& 0,
\end{eqnarray*}
i.e., $K_{\Omega\setminus E}(\cdot,z)=K_\Omega(\cdot,z)$ on $\Omega\setminus E$. Thus for any $f\in{A^2(\Omega\setminus E)}$,
\begin{equation}\label{eq:f_extension}
f(z)=\int_{\Omega\setminus E} f\overline{K_{\Omega\setminus E}(\cdot,z)}=\int_{\Omega\setminus E} f\overline{K_\Omega(\cdot,z)}=\int_\Omega \chi_{\Omega\setminus E}f\overline{K_\Omega(\cdot,z)},\ \ \ \forall\,z\in\Omega\setminus E.
\end{equation}
Here, $\chi_{\Omega\setminus E}$ denotes the chracteristic function of $\Omega\setminus E$. Since the last integral in \eqref{eq:f_extension} is the Bergman projection of $\chi_{\Omega\setminus E}f$ onto $A^2(\Omega)$, it follows that $f$ can be extended to some function in $A^2(\Omega)$. By Carleson's theorem, we conclude $\mathcal{C}_l(E)=0$.
\end{proof}

\begin{proof}[Proof of Proposition \ref{prop:volume_zero}]
Again, we have
\[
K_{\Omega\setminus E}(\cdot,z)=K_\Omega(\cdot,z)
\]
on $\Omega\setminus E$. Thus
\begin{eqnarray*}
\int_\Omega|K_\Omega(\cdot,z)|^2
&=& K_\Omega(z) = K_{\Omega\setminus E}(z) = \int_{\Omega\setminus E}|K_{\Omega\setminus E}(\cdot,z)|^2\\
&=& \int_{\Omega\setminus E}|K_\Omega(\cdot,z)|^2,\ \ \ z\in\Omega\setminus E,
\end{eqnarray*}
so that $\int_E|K_\Omega(\cdot,z)|^2=0$. Since $K_\Omega(\cdot,z)$ is not identical zero on $\Omega$, we obtain $|E|=0$.
\end{proof}

We conclude this section by proposing the following
\begin{question}
(1) Let $\Omega\subset\mathbb{C}$ be a bounded domain. Given $\varepsilon>0$, does there exists some $\gamma>0$ such that
\[
\kappa(\Omega\setminus E)-\kappa(\Omega)<\varepsilon
\]
whenever $\mathcal{C}_l(E)<\gamma$?

(2) Let $\Omega$ be a domain in $\mathbb{C}^n$ ($n\geq2$). Given $\varepsilon>0$, does there exists some $\gamma>0$ such that
\[
\kappa(\Omega\setminus E)-\kappa(\Omega)<\varepsilon
\]
whenever $E$ is a closed set in $\mathbb{C}^n$ which satisfies the conditions (1)-(3) in Theorem \ref{th:Hartogs_L2} with $|E|<\gamma$?
\end{question}

\section{Minimum of the Bergman kernel}

\begin{proof}[Proof of Theorem \ref{th:minimum_value}]
Let $g_\Omega$ denote the (negative) Green function of a domain $\Omega\subset\mathbb{C}_\infty$, where $\mathbb{C}_\infty$ is the Riemann sphere. For $z\in\Omega\setminus\{\infty\}$, we set
\[
c_\Omega(z):=\exp\left(\lim_{\zeta\rightarrow z} \left(g_\Omega(\zeta,z)-\log|\zeta-z|\right)\right),
\]
i.e.,
\[
g_\Omega(\zeta,z)=\log|\zeta-z|+\log{c_\Omega(z)}+o(1),\ \ \ \zeta\rightarrow{z}.
\]
On the other hand, if $\infty\in\Omega$, then
\[
g_\Omega(\zeta,\infty)=-\log|\zeta|+\log{\mathcal{C}_l(\mathbb{C}_\infty\setminus\Omega)}+o(1),\ \ \ \zeta\rightarrow\infty
\]
(see, e.g., Ransford \cite{Ransford}, Theorem 5.2.1). Let $\Omega\subset\mathbb{C}$. Given $z\in\Omega$, define
\[
\Phi_z(\zeta)=\frac{1}{\zeta-z},\ \ \ \zeta\in\mathbb{C}_\infty.
\]
Clearly, $\Phi_z(\Omega)$ is a domain in $\mathbb{C}_\infty$ and
\begin{eqnarray*}
g_\Omega(\zeta,z)
&=& g_{\Phi_z(\Omega)}(\Phi_z(\zeta),\Phi_z(z)) = g_{\Phi_z(\Omega)}\left(\frac{1}{\zeta-z},\infty\right)\\
&=& \log|\zeta-z|+\log{\mathcal{C}_l(\mathbb{C}_\infty\setminus{\Phi_z(\Omega)})}+o(1),\ \ \ \zeta\rightarrow{z}.
\end{eqnarray*}
Thus
\begin{equation}\label{eq:Robin_constant}
c_\Omega(z)=\mathcal{C}_l(\mathbb{C}_\infty\setminus{\Phi_z(\Omega)})=\mathcal{C}_l(\Phi_z(\mathbb{C}_\infty\setminus\Omega)).
\end{equation}
By definition, there exists a sequence $r_j\downarrow\mathscr{R}_{L,\alpha}(\Omega)$ such that
\begin{equation}\label{eq:capacity_radius_lower}
\mathcal{C}_l(\overline{\Delta(z,r_j)}\setminus\Omega)>\alpha\cdot r_j
\end{equation}
for all $z\in\Omega$. Note that
\begin{equation}\label{eq:Phi_z}
|\Phi_z^{-1}(\zeta_1)-\Phi_z^{-1}(\zeta_2)|\leq r_j^2|\zeta_1-\zeta_2|,\ \ \ \forall\,\zeta_1,\zeta_2\in\Phi_z(\overline{\Delta(z,r_j)}).
\end{equation}
We have
\begin{eqnarray}\label{eq:lower_Robin_2}
c_\Omega(z)
&\geq& {\mathcal{C}_l\left(\Phi_z\left(\overline{\Delta(z,r_j)}\setminus\Omega\right)\right)}\ \ \ \text{(by \eqref{eq:Robin_constant})}\nonumber\\
&\geq& \frac{\mathcal{C}_l(\overline{\Delta(z,r_j)}\setminus\Omega)}{r_j^2}\ \ \ \text{(by \eqref{eq:Phi_z} and Theorem 5.3.1 in \cite{Ransford})}\nonumber\\
&>& \frac{\alpha}{r_j}\ \ \ \text{(by \eqref{eq:capacity_radius_lower})}\nonumber\\
&\rightarrow& \frac{\alpha}{\mathscr{R}_{L,\alpha}(\Omega)}\ \ \ (j\rightarrow+\infty).
\end{eqnarray}
The conclusion follows immediately from \eqref{eq:lower_Robin_2} and the following estimate B{\l}ocki \cite{Blocki13}:
\[
K_\Omega(z)\geq\frac{c_\Omega(z)^2}{\pi},\ \ \ \forall\,z\in\Omega.\qedhere
\]
\end{proof}

To study the higher dimensional case, we shall make use of the following special case of Ohsawa-Takegoshi extension theorem (cf. \cite{OT87}):

\begin{theorem}\label{th:OT}
Let $\Omega\subset\mathbb{C}^n$ be a bounded pseudoconvex domain and $\mathcal{L}\subset\mathbb{C}^n$ a complex line. For any $f\in{A^2(\Omega\cap\mathcal{L})}$, there exists $F\in{A^2(\Omega)}$ such that
\[
\int_\Omega|F|^2\leq C_n\mathscr{D}(\Omega)^{2n-2}\int_{\Omega\cap\mathcal{L}}|f|^2,
\]
where $\mathscr{D}(\Omega)$ denotes the diameter of\/ $\Omega$.
\end{theorem}

As a consequence of Theorem \ref{th:OT}, we have
\begin{equation}\label{eq:OT_Bergman}
K_\Omega(z)\geq C_n^{-1}\mathscr{D}(\Omega)^{2-2n}K_{\Omega\cap\mathcal{L}}(z),\ \ \ \forall\,z\in\Omega\cap\mathcal{L}.
\end{equation}
This leads to the following

\begin{proposition}\label{prop:minimum_value_2-1}
Let $\Omega\subset\mathbb{C}^n$ be a bounded pseudoconvex domain. For each $\alpha>1$, we define
\[
\mathscr{C}_\alpha(\Omega):=\inf_{z\in\Omega}\sup_{\mathcal{L}\ni z}\frac{\mathcal{C}_l\left(\left(\mathcal{L}\cap\overline{B(z,\alpha\delta_\Omega(z))}\right)\setminus\Omega\right)}{\delta_\Omega(z)},
\]
where $\mathcal{L}$ is a complex line in $\mathbb{C}^n$. Then
\[
\kappa(\Omega)\geq\frac{C_n^{-1}\mathscr{C}_\alpha(\Omega)^2}{\pi\alpha^4\mathscr{D}(\Omega)^{2n-2}\mathscr{R}(\Omega)^2}.
\]
\end{proposition}

\begin{proof}
We may assume $\mathscr{C}_\alpha(\Omega)>0$. For any $z\in\Omega$ and $0<\varepsilon<\mathscr{C}_\alpha(\Omega)$, there exists a complex line $\mathcal{L}\ni z$ such that
\[
\frac{\mathcal{C}_l\left(\left(\mathcal{L}\cap\overline{B(z,\alpha\delta_\Omega(z))}\right)\setminus\Omega\right)}{\delta_\Omega(z)}\geq\varepsilon.
\]
Similar as the proof of Theorem \ref{th:minimum_value}, we have
\[
c_{\Omega\cap\mathcal{L}}(z)\geq\frac{\mathcal{C}_l\left(\left(\mathcal{L}\cap\overline{B(z,\alpha\delta_\Omega(z))}\right)\setminus\Omega\right)}{\alpha^2\delta_\Omega(z)^2}\geq\frac{\varepsilon}{\alpha^2\delta_\Omega(z)},
\]
so that
\[
K_{\Omega\cap\mathcal{L}}(z)\geq\frac{c_{\Omega\cap\mathcal{L}}(z)^2}{\pi}\geq\frac{\varepsilon^2}{\pi\alpha^4\delta_\Omega(z)^2}.
\]
Thus \eqref{eq:OT_Bergman} gives
\[
K_\Omega(z)
\geq\frac{C_n^{-1}\varepsilon^2}{\pi\alpha^4\mathscr{D}(\Omega)^{2n-2}\delta_\Omega(z)^2}
\geq\frac{C_n^{-1}\varepsilon^2}{\pi\alpha^4\mathscr{D}(\Omega)^{2n-2}\mathscr{R}(\Omega)^2},\ \ \ \forall\,z\in\Omega.
\]
Since $\varepsilon$ can be arbitrarily close to $\mathscr{C}_\alpha(\Omega)$, we conclude the proof.
\end{proof}

\begin{proof}[Proof of Theorem \ref{th:minimum_value_2}]
For any $z\in\Omega$, there exists a complex line $\mathcal{L}_z\ni{z}$ such that
\[
\mathcal{C}_l\left(\left(\mathcal{L}_z\cap{B(z,2\delta_\Omega(z))}\right)\setminus\Omega\right)
\geq \varepsilon_n\mathscr{V}(\Omega)\delta_\Omega(z)
\]
(cf. \cite{Chen23}, p. 1945), where $\varepsilon_n>0$ depends only on $n$. It suffices to apply Proposition \ref{prop:minimum_value_2-1} with $\alpha=2$.
\end{proof}

\section{Appendix: Maz'ya-Shubin upper estimate for planar domains}\label{sec:appendix}

Let $\Omega\subsetneq\mathbb{C}$ be a domain. Recall that
\[
\mathscr{R}_{L,\alpha} (\Omega):=\sup\left\{r>0:{\mathcal{C}_l(\overline{\Delta(z,s)}\setminus\Omega)} \leq\alpha s,\ \text{for some }z\in\Omega\text{ and any }s\in(0,r)\right\}.
\]
Then we have the following Maz'ya-Shubin type upper estimate.

\begin{theorem}\label{th:MS_upper}
For any $0<\alpha<e^{-1/2}$, there exists a constant $C_\alpha>0$ such that $\lambda_1(\Omega)\leq C_\alpha\mathscr{R}_{L,\alpha}(\Omega)^{-2}$.
\end{theorem}

\begin{proof}
Recall that
\begin{equation}\label{eq:Poincare}
\int_\Omega|\phi|^2\leq \lambda_1(\Omega)^{-1}\int_\Omega|\nabla \phi|^2,\ \ \ \forall\,\phi\in{\mathrm{Lip}_0(\Omega)},
\end{equation}
where $\mathrm{Lip}_0(\Omega)$ denotes the set of Lipschitz continuous functions with a compact support in $\Omega$. Given $0<r<\mathscr{R}_{L,\alpha}(\Omega)$, there exists some $z_0\in\Omega$ with $\mathcal{C}_l(\overline{\Delta(z_0,r)}\setminus\Omega)\leq\alpha r$. For simplicity, we may assume $z_0=0$. For any $\varepsilon>0$, we take a neighbourhood $V_\varepsilon$ of $\overline{\Delta(0,r)}\setminus\Omega$ with $\overline{V}_\varepsilon\subset\Delta(0,2r)$ and $\mathcal{C}_l(\overline{V}_\varepsilon)\leq (1+\varepsilon)\alpha r$, in view of Choquet's theorem. For $R>2r$, the Green capacity $\mathcal{C}_g(\overline{V}_\varepsilon,\Delta(0,R))$ satisfies
\[
\mathcal{C}_g(\overline{V}_\varepsilon,\Delta(0,R))\leq\frac{\mathcal{C}_l(\overline{V}_\varepsilon)}{R-2r}\leq \frac{(1+\varepsilon)\alpha r}{R-2r}.
\]
Let $\mu_\varepsilon$ and $p_{\mu_\varepsilon}$ be the corresponding equilibrium measure and Green potential, respectively. It follows that
\begin{equation}\label{eq:Green_potential}
I(\mu_\varepsilon)\leq-\log\frac{R-2r}{(1+\varepsilon)\alpha r}
\end{equation}
provided $0<\alpha<1$ and $0<\varepsilon\ll 1$. Set
\[
\chi_\varepsilon=\frac{p_{\mu_\varepsilon}}{I(\mu_\varepsilon)},
\]
which satisfies $\chi_\varepsilon=1$ on $V_\varepsilon$, $0\leq\chi_\varepsilon\leq1$ on $\Delta(0,R)$, and $\chi_\varepsilon=0$ on $\partial\Delta(0,R)$. Let $0<r_1<r_2<r$ and $\eta$ be a cut-off function with $\eta=1$ on $\Delta(0,r_1)$, $\eta=0$ outside $\Delta(0,r_2)$ and $|\nabla\eta|\leq 2(r_2-r_1)^{-1}$ on $\overline{\Delta(0,r_2)}\setminus\Delta(0,r_1)$. Then
\[
\phi:=(1-\chi_\varepsilon)\eta\in\mathrm{Lip_0(\Omega)}.
\]
Since $\chi_\varepsilon$ is a harmonic function when $0<\chi_\varepsilon<1$, we infer from Green's formula that
\begin{eqnarray*}
\int_{\Delta(0,r)}|\nabla \phi|^2
&=& \int_{\Delta(0,r)}\left((1-\chi_\varepsilon)^2|\nabla\eta|^2-2(1-\chi_\varepsilon)\eta\nabla\chi_\varepsilon\cdot\nabla\eta+\eta^2|\nabla\chi_\varepsilon|^2\right)\\
&=& \int_{\Delta(0,r)}\left((1-\chi_\varepsilon)^2|\nabla\eta|^2-(1-\chi_\varepsilon)\nabla\chi_\varepsilon\cdot\nabla\eta^2+\eta^2|\nabla\chi_\varepsilon|^2\right)\\
&=& \int_{\Delta(0,r)}\left((1-\chi_\varepsilon)^2|\nabla\eta|^2+\eta^2\nabla\left((1-\chi_\varepsilon)\nabla\chi_\varepsilon\right)+\eta^2|\nabla\chi_\varepsilon|^2\right)\\
&=& \int_{\Delta(0,r)}(1-\chi_\varepsilon)^2|\nabla\eta|^2\\
&\leq& \frac{4\pi r^2}{(r_2-r_1)^2}.
\end{eqnarray*}
On the other hand, we have
\[
\int_{\Delta(0,r)}|\phi|^2\geq\int_{\Delta(0,r_1)}(1-\chi_\varepsilon)^2\geq\frac{1}{\pi r_1^2}\left(\int_{\Delta(0,r_1)}(1-\chi_\varepsilon)\right)^2
\]
in view of the Cauchy-Schwarz inequality. These combined with \eqref{eq:Poincare} yield
\begin{equation}\label{eq:lambda_1}
\left(\pi r_1^2 - \int_{\Delta(0,r_1)}\chi_\varepsilon\right)^2\leq \frac{4\pi^2r^2r_1^2}{(r_2-r_1)^2}\lambda_1(\Omega)^{-1}.
\end{equation}

It remains to find an upper bound for
\begin{eqnarray*}
\int_{\Delta(0,r_1)}\chi_\varepsilon
&=& \frac{1}{I(\mu_\varepsilon)}\int_{z\in\Delta(0,r_1)}\int_{w\in\Delta(z_0,R)}g_{\Delta(0,R)}(z,w)d\mu_\varepsilon(w)\\
&=& \frac{1}{I(\mu_\varepsilon)}\int_{w\in\Delta(0,R)}\left(\int_{z\in\Delta(0,r_1)}g_{\Delta(0,R)}(z,w)\right)d\mu_\varepsilon(w)
\end{eqnarray*}
Given $0<s<r_1$, define
\[
F(s):=\int_{z\in\partial\Delta(0,s)}g_{\Delta(0,R)}(z,w)|dz|=s\int^{2\pi}_0g_{\Delta(0,R)}(se^{i\theta},w)d\theta.
\]
Note that
\[
g_{\Delta(0,R)}(z,w)=\log\frac{R|z-w|}{|R^2-z\overline{w}|},
\]
and
\[
\Delta_z g_{\Delta(0,R)}(z,w)=2\pi\delta_w,
\]
where $\Delta_z$ denotes the Laplacian operator with respect to the variable $z$ and $\delta_w$ is the Dirac measure at $w$. The Poisson-Jensen formula states that
\[
u(0)=\frac{1}{2\pi}\int^{2\pi}_0 u(se^{i\theta})d\theta + \frac{1}{2\pi}\int_{\zeta\in\Delta(0,s)}\log\frac{|\zeta|}{s}\Delta u
\]
holds if $u$ is a subharmonic function on $\overline{\Delta(0,s)}$. Take $u=g_{\Delta(0,R)}(\cdot,w)$. If $|w|>s$, then $\Delta u=0$ on $\Delta(0,s)$, so that
\[
F(s)=2\pi{s} u(0)=2\pi{s} \log\frac{|w|}{R}.
\]
If $|w|<s$, then $\Delta u=2\pi\delta_w$ on $\Delta(0,s)$, and hence
\[
F(s)=2\pi{s} u(0) - 2\pi{s}\log\frac{|w|}{s}
=2\pi{s} \log\frac{s}{R}.
\]
As a consequence, if $|w|>r_1$, then
\begin{eqnarray}\label{eq:lower_1}
\int_{z\in\Delta(0,r_1)}g_{\Delta(z_0,R)}(z,w)
&=& \int^{r_1}_0F(s)ds = \int^{r_1}_0 2\pi{s} \log\frac{|w|}{R}ds\notag\\
&\geq& -\log\frac{R}{r_1}\int^{r_1}_0 2\pi{s} ds;
\end{eqnarray}
while if $|w|\leq r_1$, then
\begin{eqnarray*}
\int_{z\in\Delta(0,r_1)}g_{\Delta(0,R)}(z,w)
&=& \int^{r_1}_0F(s)ds\\
&=& \int^{|w|}_0 2\pi{s} \log\frac{|w|}{R}ds + \int^{r_1}_{|w|}2\pi{s} \log\frac{s}{R}ds.
\end{eqnarray*}
Set
\[
\varphi(t):=\int^{t}_0 2\pi{s} \log\frac{t}{R}ds + \int^{r_1}_{t}2\pi{s} \log\frac{s}{R}ds.
\]
Since
\[
\varphi'(t)=\frac{1}{t}\int^t_02\pi{s}ds\geq0,
\]
we have
\[
\varphi(0)\leq\varphi(r_1)=-\log\frac{R}{r_1}\int^{r_1}_0 2\pi{s} ds.
\]
This combined with \eqref{eq:lower_1} gives
\begin{eqnarray}\label{eq:lower_2}
\int_{z\in\Delta(0,r_1)}g_{\Delta(0,R)}(z,w)
&\geq& \varphi(0) = \int^{r_1}_0 2\pi{s} \log\frac{s}{R}ds\notag\\
&=& -\pi{r_1^2}\left(\frac{1}{2}+\log\frac{R}{r_1}\right),\ \ \ \forall\,w\in\Delta(0,R).
\end{eqnarray}
Since $\mu_\varepsilon$ is a probability measure and $I(\mu_\varepsilon)<0$, we have
\[
\int_{\Delta(0,r_1)}\chi_\varepsilon
\leq -\frac{\pi{r_1^2}}{I(\mu_\varepsilon)}\left(\frac{1}{2}+\log\frac{R}{r_1}\right)
\leq  \pi{r_1^2}\frac{\frac{1}{2}+\log\frac{R}{r_1}}{\log\frac{R-2r}{(1+\varepsilon)\alpha r}}
\]
in view of \eqref{eq:Green_potential}. Therefore, if $0<\alpha<e^{-1/2}$, we can choose $\varepsilon$ so small that $e^{1/2}(1+3\varepsilon)\alpha<1$. If we set $r_1=e^{1/2}(1+2\varepsilon)\alpha r$ and $r_2=e^{1/2}(1+3\varepsilon)\alpha r$, then there exists $R=Nr$ with $N\gg1$ such that
\[
\frac{1}{2}+\log\frac{R}{r_1}=\log\frac{e^{1/2}R}{r_1}=\log\frac{N}{(1+2\varepsilon)\alpha }<\log\frac{N-2}{(1+\varepsilon)\alpha }=\log\frac{R-2r}{(1+\varepsilon)\alpha r}.
\]
Choose certain $\varepsilon$ and $N=R/r$ depending only on $\alpha$, we obtain
\[
\int_{\Delta(0,r_1)}\chi_\varepsilon \leq c_\alpha \pi r_1^2
\]
for some $0<c_\alpha<1$, so that \eqref{eq:lambda_1} gives
\[
\lambda_1(\Omega)\leq\frac{4r^2}{(1-c_\alpha)^2r_1^2(r_2-r_1)^2} =\frac{C_\alpha}{r^2}
\]
for some $C_\alpha>0$. It suffices to let $r\rightarrow\mathscr{R}_{L,\alpha}(\Omega)$.
\end{proof}

\end{document}